\documentclass{amsart}
\usepackage{multicol}
\usepackage{enumerate}
\usepackage{mathtools}
\usepackage{amsfonts}
\usepackage{amssymb}
\usepackage{mathrsfs}
\usepackage{graphicx}
\usepackage{url}
\usepackage{microtype}
\usepackage[margin=1in]{geometry}
\newtheorem{thm}{Theorem}[section]
\newtheorem{prop}[thm]{Proposition}

\newtheorem{cor}[thm]{Corollary}

\newtheorem{lemma}[thm]{Lemma}
\newtheorem*{thm*}{Theorem}						
\newtheorem*{prop*}{Proposition}
\newtheorem*{lemma*}{Lemma}
\newtheorem*{cor*}{Corollary}
\newtheorem*{conj*}{Conjecture}

\theoremstyle{definition}
\newtheorem{definition}[thm]{Definition}

\newtheorem{remark}[thm]{Remark}

\newtheorem{notation}[thm]{Notation}

\newcommand{\R}{\mathbb{R}}
\newcommand{\C}{\mathbb{C}}

\renewcommand{\vec}[1]{\mathbf{#1}}
\renewcommand{\P}{\mathbb{P}}

\let\hom\relax 
\DeclareMathOperator{\hom}{hom}
\DeclareMathOperator{\Hom}{Hom}

\DeclareMathOperator{\Ext}{Ext}
\DeclareMathOperator{\ext}{ext}

\DeclareMathOperator{\pic}{Pic}
\let\O\relax
\DeclareMathOperator{\O}{\mathcal{O}}
\DeclareMathOperator{\ch}{ch}
\DeclareMathOperator{\rk}{rk}

\DeclareMathOperator{\knum}{K_{num}}

\numberwithin{equation}{thm}

\usepackage{color}
\usepackage[dvipsnames]{xcolor}

\makeatletter
\let\c@theorem\c@figure
\makeatother
\begin{document}

\title[Stability conditions for restrictions of vector bundles on projective surfaces]{Stability conditions for restrictions of vector bundles on projective surfaces}

\author[J. Kopper]{John Kopper}
\address{Department of Mathematics, The Pennsylvania State University, University Park, PA, 16802}
\email{kopper@psu.edu}

\thanks{During the preparation of this article the author was partially supported by NSF RTG grant DMS-1246844.}


\begin{abstract}
Using Bridgeland stability conditions, we give sufficient criteria for a stable vector bundle on a smooth complex projective surface to remain stable when restricted to a curve. We give a stronger criterion when the vector bundle is a general vector bundle on the plane. As an application, we compute the cohomology of such bundles for curves that lie in the plane or on Hirzebruch surfaces.
\end{abstract}

\maketitle

\section{Introduction}
In this paper we give sufficient criteria for a stable bundle on a smooth complex projective surface to remain stable when restricted to a curve. The main results in this subject are due to Flenner \cite{flenner} and Mehta-Ramanathan \cite{mehta-ramanathan} who give criteria for restrictions of bundles to remain stable on divisors and complete intersections. In the case of a surface, Flenner's theorem becomes:

\begin{thm*}[Flenner]
Let $(X,H)$ be a smooth polarized surface. If $E$ is a $\mu_H$-semistable bundle of rank $r$ on $X$ and $C\subset X$ is a general curve of class $dH$, then $E|_C$ is semistable if
\[
\frac{d+1}{2} > H^2 \max\left\{ \frac{r^2-1}{4}, 1\right\}.
\]
\end{thm*}

Bogomolov gave a more precise restriction theorem for surfaces \cite{bogomolov} (see also \cite{huybrechts-lehn}). Notably, Bogomolov's result applies to \emph{any} smooth curve moving in an ample class. For a vector bundle $E$, let
\[
\Delta(E) =\frac{1}{2} \cdot \frac{c_1(E)^2}{r^2} - \frac{ch_2(E)}{r}.
\]
Then we have:

\begin{thm*}[Bogomolov]
Let $(X,H)$ be a smooth polarized surface. If $E$ is $\mu_H$-stable bundle of rank $r$ on $X$ and $C \subset X$ is a smooth curve of class $dH$, then $E|_C$ is stable if
\[
2d > \binom{r}{\left\lfloor\frac{r}{2}\right\rfloor}\binom{r-2}{\left\lfloor\frac{r}{2}\right\rfloor - 1}r\Delta(E) + 1,
\]
\end{thm*}

Langer recently gave a very strong restriction theorem which holds for very ample divisors in higher dimensions and arbitrary characteristic \cite[Thm. 0.1]{langer2}. In the case of a surface in characteristic zero, Langer's theorem becomes:

\begin{thm*}[Langer]
Let $X$ be a smooth surface and $H$ a very ample divisor on $X$. If $E$ is a $\mu_H$-semistable bundle of rank $r \geq 2$ and $C \subset X$ a general curve of class $dH$, then $E|_C$ is semistable if
\[
\frac{(d+1)(d+2)}{2} > \left( \max \left\{\frac{r^2-1}{4},1\right\}H^2+1\right)r^2\Delta(E) + 1.
\]
\end{thm*}

More recent developments in stability conditions for derived categories have led to new results in the study of restrictions of bundles on surfaces. For example, Feyzbakhsh \cite{feyzbakhsh} used the machinery of stability conditions developed by Bridgeland \cite{bridgeland1} \cite{bridgeland2} to give an effective criterion to guarantee the stability of restrictions of bundles on K3 surfaces. We follow the method of Feyzbakhsh to give a restriction theorem for all surfaces. When the surface is the projective plane, we give a much stronger stability criterion.

A notable difference between our first result and Langer's theorem is that we are able to replace the requirement that $C$ be ample with a weaker condition on the positivity of $C$. An application of this improvement is suggested by Proposition \ref{prop:extending_vb} in which we construct a (rational) map between moduli spaces: given a smooth polarized surface $(X,H)$ and curve $C \subset X$, under the conditions of Theorem \ref{thm:restriction_allsurfaces}, the restriction map defines a morphism from the space of $\mu_H$-stable sheaves on $X$ to the space of stable sheaves on $C$. The precise statements of the main theorems are as follows:
\begin{thm*}[\ref{thm:restriction_allsurfaces}]
Suppose $(X,H)$ is a smooth polarized surface, $C$ an integral curve on $X$, and $E$ a $\mu_H$-stable vector bundle of rank $r \geq 2$ on $X$. Then $E|_C$ is stable if
\[
\frac{C^2}{2H \cdot C} > r(r-1)\Delta(E)+\frac{1}{2r(r-1)H^2}.
\]
\end{thm*}

\begin{remark}
During the preparation of this paper, S. Feyzbaksh pointed out that the proof of Theorem \ref{thm:restriction_allsurfaces} relies only on the notion of \emph{tilt stability} and a direct generalization therefore holds when restricting stable bundles to divisors on higher dimensional varieties. A proof of this fact in the case the divisor moves in the same ample class as the polarization can be found in Feyzbakhsh's thesis \cite{feyzbakhsh-thesis}. The author wants to thank S. Feyzbakhsh for her thoughtful conversations.
\end{remark}

In the special case that $X$ is the plane and $E$ is general in moduli, we obtain a much stronger bound for the degree of $C$ that does not grow with the rank of $E$.

\begin{thm*}[\ref{thm:restriction_p2}]
Suppose $C \subset \P^2$ is a degree $d$ integral curve and $E$ is a general slope stable vector bundle on $\P^2$ with Chern character $\vec{v}$ such that $\Delta(\vec{v}) \geq 4$. Then $E|_C$ is stable if
\[
d > \sqrt{5+8\Delta(E)} + 5.
\]
\end{thm*}

Following Feyzbakhsh \cite{feyzbakhsh}, we prove both theorems using stability conditions for surfaces as constructed by Bridgeland \cite{bridgeland2}, Arcara-Bertram \cite{arcara-bertram}, and Toda \cite{toda}. If $E$ is a slope stable sheaf on a surface $X$ and $i:C\hookrightarrow X$ a  curve, then $E$ and $i_\ast E|_C$ fit into a distinguished triangle
\[
E \to i_\ast E|_C \to E(-C)[1] \to E[1]
\]
in the derived category $\mathcal{D}^b(X)$, where $E(-C)[1]$ denotes the complex associated to $E(-C)$ with the grading shifted by one. If $E$ and $E(-C)[1]$ are stable of the same phase, then $i_\ast E|_C$ is semistable. We will show that in these circumstances $E|_C$ is a slope semistable sheaf on $C$ (Lemma \ref{lemma:bridgelandstable_implies_slopestable}). By slightly perturbing the stability condition, we will show that $E|_C$ is in fact stable (Lemma \ref{lem:stableimplies_stableres}).

The main content of the argument is producing sufficient criteria to ensure that the conditions of Lemmas \ref{lemma:bridgelandstable_implies_slopestable} and \ref{lem:stableimplies_stableres} hold. Specifically, we need $E$ and $E(-C)[1]$ to be stable of the same phase. For a fixed Chern character $\vec{v}$, there is a wall-and-chamber decomposition of the \emph{stability manifold} parametrizing stability conditions such that stable objects with Chern character $\vec{v}$ can only be destabilized when a wall is crossed \cite{bridgeland1}. To prove Theorem \ref{thm:restriction_allsurfaces}, we estimate the \emph{Gieseker wall}, a wall in the stability manifold bounding the chamber consisting of stability conditions $\sigma$ for which the only $\sigma$-semistable objects of Chern character $\vec{v}$ are the Gieseker semistable sheaves. Similarly, to prove Theorem \ref{thm:restriction_p2} we estimate the \emph{effective wall} for sheaves on $\P^2$. The effective wall bounds the chamber in which the \emph{general} Gieseker semistable sheaf $E$ is $\sigma$-semistable. The Giesker wall was computed explicitly in \cite{ch-nef} and the effective wall for sheaves on $\P^2$ was computed in \cite{chw}. We expect the techniques used in proving Theorem \ref{thm:restriction_p2} to generalize to any surface for which the effective wall is known.

\subsection*{Structure of this paper} In Section \ref{sec:preliminaries} we recall necessary facts about stable sheaves and stability conditions for surfaces. In Section \ref{sec:all_surfaces} we prove our restriction theorems. We conclude in Section \ref{sec:cohomology} by giving some applications of restriction theorems to Brill-Noether problems for curves in the plane and on Hirzebruch surfaces.

\subsection*{Acknowledgments} The author is grateful to Izzet Coskun, Takumi Murayama, Tim Ryan, and Matthew Woolf for their helpful feedback and input. The author is also grateful to the anonymous referees for their thoughtful responses.

\section{Preliminaries}\label{sec:preliminaries}
In this section we will review facts about stable sheaves and stability conditions that will be used throughout the paper. All schemes are defined over the field of complex numbers.

\subsection{Moduli of stable sheaves} For more details on moduli spaces of sheaves, we refer the reader to \cite{huybrechts-lehn}. Let $(X,H)$ be a polarized variety and $E$ a coherent sheaf on $X$ of pure dimension $d=\dim(X)$. Then we may write
\[
P(E,m) = \chi(E\otimes\O_X(m)) = \sum_{i=0}^{d} \alpha_i(E) \frac{m^i}{i!}
\]
for unique integers $\alpha_i(E)$. The \emph{degree} $\deg(E)$ is defined as the number $\alpha_{d-1}(E) - r\cdot\alpha_{d-1}(\O_X)$. If $X$ is smooth or an integral curve, then $\deg(E) = c_1(E) \cdot H^{d-1}$ by Hirzebruch-Riemann-Roch \cite[Ex. IV.1.9]{hartshorne}.

The \emph{reduced Hilbert polynomial} of $E$ is defined as
\[
p(E,m) = \frac{p(E,m)}{\alpha_d(E)}.
\]
We say that $E$ is \emph{Gieseker semistable} if for all proper subsheaves $F \subset E$ we have $p(F,m) \leq p(E,m)$ for $m \gg 0$. We say $E$ is \emph{Gieseker stable} if strict inequality holds.

If $E$ has rank $r > 0$, then we define the $\mu_H$-\emph{slope} of $E$ to be the number
\[
\mu_H(E) = \frac{\deg(E)}{rH^d}.
\]
We say that a torsion-free sheaf $E$ is $\mu_H$-\emph{semistable} (or \emph{slope semistable}) if $\mu_H(F) \leq \mu_H(E)$ for all proper subsheaves $F$ of rank strictly less than $r$. The sheaf $E$ is said to be $\mu_H$-\emph{stable} (or \emph{slope stable}) if strict inequality holds. Note that the notions of slope stability and Gieseker stability both depend on the choice of ample divisor $H$.


Write $c(E) = 1 + c_1(E) + c_2(E) + \cdots c_d(E) = \prod_{i=1}^d (1+\gamma_i(E))$, with $\gamma_i(E) \in A^1(X)$ for all $i$. Then the \emph{Chern character} of $E$ is the polynomial
\[
\ch(E) = \sum_{i=1}^d \exp(\gamma_i) = \rk(E) + c_1(E) + \frac{1}{2}(c_1(E)^2 - 2c_2(E)) +\cdots.
\]
The degree $n$ term of $\ch(E)$ is denoted $\ch_n(E)$. We extend the definition of the Chern character (and consequently that of the slope) to the derived category $\mathcal{D}^b(X)$ via
\[
\ch(\cdots \rightarrow E^{-1} \to E^0 \to E^1 \to \cdots) = \sum_{i=-\infty}^\infty (-1)^i \ch E^{i}.
\]
Note that there are finitely many terms in this sum because we are working in the bounded derived category. Let $E$ be an object in the derived category of $X$, not necessarily a sheaf. Write $\ch(E) = (\ch_0,\ch_1,\ch_2,\dots, \ch_d)$. We will refer to $\ch_0$ as the \emph{rank of $E$} though this number may be negative. If $X$ is a surface, then define the \emph{discriminant} of $E$ to be the number
\[
\Delta(E) = \frac{1}{2} \cdot \frac{\ch_1^2}{\ch_0^2} - \frac{\ch_2}{\ch_0}.
\]
When working with Bridgeland stability conditions it is convenient to make a slight change of coordinates. For any $\mathbb{Q}$-divisor $D$, define the \emph{twisted Chern character} $\ch^D = \exp(-D)\ch$. Explicitly,
\[
\ch^D_0 = \ch_0, \qquad \ch^D_1 = \ch_1 - D \ch_0, \qquad \ch^D_2 = \ch_2 - D \ch_1 + \frac{D^2}{2} \ch_0.
\]

We then define the \emph{twisted slope} and \emph{twisted discriminant}, respectively:
\[
\mu_{H,D}(E) =  \begin{cases} \displaystyle
      \frac{H \cdot \ch_1^D}{H^2\ch_0^D} & \ch_0^D(E) > 0 \\
      \infty & \ch_0^D(E) = 0 
   \end{cases}
, \qquad \Delta_{H,D}(E) = \frac{1}{2}\mu_{H,D}(E)^2 - \frac{\ch_2^D}{H^2\ch_0^D}.
\]
Note that $\mu_{H,D}(E)$ differs from $\mu_H(E) = \mu_{H,0}(E)$ by a constant, and so $E$ is a $\mu_H$-stable (resp. semistable) sheaf if and only if $\mu_{H,D}(F) < \mu_{H,D}(E)$ (resp. $\mu_{H,D}(F) \leq \mu_{H,D}(E)$) for all $D$ and all proper subsheaves $F\subset E$ with rank less than $\ch_0(E)$.

Finally, define the \emph{reduced twisted Hilbert polynomial} of a positive rank sheaf $E$ by
\[
p^{E}_{H,D}(m) = \frac{\chi(E \otimes \O_X(mH - D))}{\ch_0^D(E)},
\]
where the Euler characteristic is defined formally by Riemann-Roch. We say $E$ is \emph{$(H,D)$-twisted Gieseker semistable} if for every nonzero proper subsheaf $F \subset E$, $p^F_{H,D}(m) \leq p^E_{H,D}(m)$ for $m \gg 0$. We say $E$ is \emph{$(H,D)$-twisted Gieseker stable} if strict inequality holds.

\begin{lemma}\label{lemma:change_of_deltas}
Let $E$ be a vector bundle. Then $H^2\Delta_{H,D}(E) \geq \Delta(E)$, and there exists a $\mathbb{Q}$-divisor $D$ such that $H^2\Delta_{H,D}(E) = \Delta(E)$.
\end{lemma}
\begin{proof}
Write $\ch_1(E) = eH + \varepsilon$ and $D = dH + \delta$, where $H \cdot \varepsilon = H \cdot \delta = 0$. Then direct computation shows:
\[
H^2\Delta_{H,D}(E) - \Delta(E) = \frac{-1}{2\ch_0(E)^2}(\ch_0(E)\delta - \varepsilon)^2.
\]
This quantity is nonnegative by the Hodge index theorem and equals zero precisely when $\delta = \varepsilon/\ch_0(E)$.
\end{proof}

Note that we can take $d=0$ in Lemma \ref{lemma:change_of_deltas} to obtain a divisor $D$ such that $H^2\Delta_{H,D}(E) = \Delta(E)$ and $H \cdot D = 0$.

A famous theorem of Bogomolov says that if $X$ is a smooth projective surface and $E$ is slope semistable, then $\Delta(E) \geq 0$ \cite{bogomolov}. By the above lemma, we also have $\Delta_{H,D}(E) \geq 0$.

For a given Chern character $\vec{v}$, Matsuki-Wentworth \cite{matsuki-wentworth} showed that there are projective moduli spaces parametrizing S-equivalence classes of $(H,D)$-twisted semistable torsion-free sheaves with Chern character $\vec{v}$. We will denote this space by $M_{X,(H,D)}(\vec{v})$. When $X = \P^2$ we will always choose $H$ to be the class of a line. In fact, if the Picard rank of $X$ is one, then $M_{X,(H,D)}(\vec{v})$ does not depend on $(H,D)$. We will supress the subscript $X,(H,D)$ whenever it is clear from context.

In order to compare sheaves on the surface $X$ to their restrictions to a curve $C \subset X$, we will need to know the Chern character of pushforwards of sheaves on $C$. The following standard lemma can be immediately verified in the case that $C$ is smooth by using the Grothendieck-Riemann-Roch formula.

\begin{lemma}\label{lemma:ch_of_pushforward}
Let $C$ be an integral curve on a smooth surface $X$, and let $F$ be a sheaf of rank $r$ on $C$. Then
\[
\ch(i_\ast F) = \left(0, rC, \deg(F) - \frac{rC^2}{2}\right).
\]
\end{lemma}
\begin{proof}
First we assume $F = \O_C$. We then show the claim when $F$ is any rank one sheaf, then we induct on the rank $r$ of $F$. If $F = \O_C$, then we have the exact sequence
\[
0 \to \O_X(-C) \to \O_X \to i_\ast \O_C \to 0,
\]
where $i:\hookrightarrow X$ is the inclusion. By additivity, $\ch(i_\ast \O_C) = \ch(\O_X) - \ch(\O_X(-C)) = (0,C,-C^2/2)$ as desired. Suppose now the claim holds for $\O(D)$ and $p \in C$ is a smooth point. We have the exact sequence
\[
0 \to \O_C(D-p) \to \O_C(D) \to \C_p \to 0.
\]
After pushing forward, we have
\[
\ch(i_\ast\O_C(D-p)) = \ch(i_\ast \O_C(D)) - \ch(i_\ast \C_p) = (0,C,\deg(D-p) - C^2/2).
\]
A similar argument shows that the claim holds for $\O_C(D+p)$. Since any line bundle can be obtained from $\O_C$ by adding and subtracting smooth points, the claim must hold for any line bundle.

Now let $F$ be any rank one sheaf. If $L$ is an ample line bundle, then $F\otimes L^{\otimes m}$ is globally generated for $m \gg 0$. In particular, $F \otimes L^{\otimes m}$ has a global section $s$. The map $\O_C \to F\otimes L^{\otimes m}$ defined by multiplication by $s$ is injective. Untwisting, we have an injection $(L^\vee)^{\otimes m} \to F$, and therefore an exact sequence
\[
0 \to M \to F \to Q \to 0
\]
where $M$ is a line bundle and $Q$ is supported in dimension zero. Pushing forward, $\ch (i_\ast F) = \ch(i_\ast M) + \ch(i_\ast Q)$, and $\ch (i_\ast Q) = (0,0,\chi(Q))$. All line bundles $M$ have been shown to satisfy the claim, thus we have $\ch (i_\ast F) = (0,C,\deg(M) + \chi(Q) - C^2/2)$.

We show that $\deg(F) = \deg(M) + \chi(Q)$. The Hilbert polynomial $P(F,m)$ of $F$ is given by $P(F,m) = \alpha_1(F) m + \alpha_0(F)$ for some coefficients $\alpha_1, \alpha_0$. Similarly, $P(M,m) = \alpha_1(M)m + \alpha_0(M)$, $P(Q,m) = \alpha_0(Q)$. By the additivity of $P$, $\alpha_0(F) = \alpha_0(M) + \alpha_0(Q)$ and $\alpha_1(F) = \alpha_1(M)$. By definition, $\deg(F) = \alpha_0(F) - \alpha_0(\O_C) = \deg(M) + \alpha_0(Q)$. Since $Q$ is supported at a point, $\alpha_0 (Q) = \chi(Q)$. Thus the claim holds for any rank one sheaf $F$.

Finally, we induct on $r$. If $F$ is any sheaf of rank $r>1$, then $F$ fits into an exact sequence
\[
0 \to M \to F \to Q \to 0
\]
where $M$ is a line bundle and $Q$ has rank $r-1$. By the induction hypothesis and the additivity of $\ch$, the claim holds for $F$.
\end{proof}

\subsection{Stability conditions on surfaces}

\begin{definition}{\cite{bridgeland1}}
Let $\mathcal{D}^b(X)$ denote the bounded derived category of coherent sheaves on a projective variety $X$. A \emph{stability condition} $\sigma$ on $\mathcal{D}^b(X)$ is a pair $\sigma = (Z, \mathcal{A})$, where $Z:K_0(X) \to \C$ is a group homomorphism and $\mathcal{A}$ is the heart of a bounded $t$-structure on $\mathcal{D}^b(X)$ satisfying three properties:
\begin{enumerate}
	\item (Positivity) $Z$ maps nonzero objects in $\mathcal{A}$ to the extended upper half-plane $\mathbb{H} = \{Re^{i\theta} : \theta\in(0,\pi], R>0\}$.
	\item (Harder-Narasimhan filtrations) For an object $E$ of $\mathcal{A}$, define the $\sigma$-slope of $E$ as
	\[
	\mu_\sigma(E) = -\frac{\Re(Z(E))}{\Im(Z(E))}.
	\]
	We call $E$ \emph{$\sigma$-stable} (resp. \emph{semistable}) if for every proper subobject $F$ of $E$ we have $\mu_\sigma(F) < \mu_\sigma(E)$ (resp. $\leq$). The pair $(\mathcal{A},Z)$ must satisfy the \emph{Harder-Narasimhan property}: for every object $E$ of $\mathcal{A}$, there is a finite filtration
	\[
	0 = E_0 \subset E_1 \subset \cdots \subset E_n = E,
	\]
	such that $E_i/E_{i-1}$ is $\mu_\sigma$-semistable, and $\mu_\sigma(E_i/E_{i-1}) > \mu_\sigma(E_{i+1}/E_i)$ for all $i$.
	\item (Support property) Fix a norm $\lVert \cdot \rVert$ on $K_{\text{num}}(X) \otimes \R$. Then there must exist a constant $C > 0$ such that
	\[
	\lVert E \rVert \leq C \lVert Z(E) \rVert
	\]
	for all semistable objects $E$ in $\mathcal{A}$.
\end{enumerate}
\end{definition}

In the case when $X$ is a smooth surface, Bridgeland \cite{bridgeland2}, Arcara-Bertram \cite{arcara-bertram}, and Toda \cite{toda} explicitly constructed stability conditions. Let $H$ be an ample divisor on $X$ and $D$ any $\mathbb{Q}$-divisor. For $s \in \R$, define the following subcategories of $\mathrm{Coh}(X)$:
\begin{align*}
&\mathcal{Q}_s = \{Q \in \mathrm{Coh}(X): Q \text{ is torsion or } \mu_{H,D}(Q') > s \text{ for all quotients }Q'\text{ of }Q\}\\
&\mathcal{F}_s = \{F \in \mathrm{Coh}(X) : F \text{ is torsion-free and } \mu_{H,D}(F') \leq s \text{ for all subsheaves }F' \text{ of }F \}.
\end{align*}
We then define the full subcategory $\mathcal{A}_s$ of $\mathcal{D}^b(X)$ as
\[
\mathcal{A}_s = \{F^\bullet \in \mathcal{D}^b(X) : \mathcal{H}^{-1}(F^\bullet) \in \mathcal{F}_s, \mathcal{H}^0(F^\bullet) \in \mathcal{Q}_s, \mathcal{H}^i(F^\bullet) =0 \text{ for } i \neq -1,0\}.
\]
Next, for $E$ an object in $\mathcal{D}^b(X)$, $s,t\in \R$, define.
\[
Z_{s,t}(E) 		= -\ch_2^{D+sH}(E) + \frac{t^2H^2}{2} \ch_0^{D + sH}(E) + iH \cdot \ch_1^{D+sH}(E).
\]
Then the pair $(Z_{s,t},\mathcal{A}_s)$ defines a stability condition on $\mathcal{D}^b(X)$ when $t > 0$ \cite{arcara-bertram}, and if $\ch_0(E) > 0$, we have:
\[
\mu_{\sigma}(E)	= \frac{(\mu_{H,D}(E)-s)^2 - t^2 - 2\Delta_{H,D}(E)}{\mu_{H,D}(E) - s}.
\]

If $E$ is a $\mu_{H,D}$-stable torsion-free sheaf on $X$ and $s < \mu_{H,D}(E)$, then $E$ is in the category $\mathcal{A}_s$. Similarly, the object $E(-C)[-1]$ is given by the complex $\cdots \to 0 \to E(-C) \to 0 \to \cdots$, concentrated in degree $-1$, so $E(-C)[1]$ is in $\mathcal{A}_s$ if and only if $E(-C)$ is in $\mathcal{F}_s$. Equivalently, we must have $s\geq \mu_{H,D}(E(-C)) = \mu_{H,D}(E) - C \cdot H/H^2$.

\begin{notation}
If $(X,H)$ is a smooth polarized surface and $\sigma=(Z_{s,t}, \mathcal{A}_s)$ is a stability condition, then we write the $\sigma$-slope as $\mu_{s,t} = \mu_\sigma$. We identify the family of stability conditions of the form $(Z_{s,t}, \mathcal{A}_s)$ with the half-plane $\{(s,t): s,t\in \R,\, t>0\}$ called the \emph{$(H,D)$-slice}. 
\end{notation}

An important feature of stability conditions is that for a fixed Chern character there is a wall-and-chamber decomposition of the space of stability conditions. A \emph{virtual wall} in the $(H,D)$-slice is a set of points of the form
\[
W(E,F) = \{(s,t) : \mu_{s,t}(E) = \mu_{s,t}(F)\}
\]
If $E$ is an object of $\mathcal{D}^b(X)$ such that $E$ is stable for $\sigma = (Z_{s,t},\mathcal{A}_s)$ but not for $\sigma' = (Z_{s',t'},\mathcal{A}_{s'})$, then there is a wall $W(E,F)$ separating the points $(s,t)$ and $(s',t')$ such that $E$ is stable for nearby points on one side of the wall but not for points on the other side. We call such walls \emph{actual walls}. The actual walls in the $(H,D)$-slice are nested semicircles with bounded centers and radii \cite{abch} (see \cite{maciocia} for the general case).

If $E$ and $F$ are objects of $\mathcal{D}^b(X)$ of nonzero rank, then the wall $W(E,F)$ has center $s_0$ and radius $\rho_0$ given by:
\[
s_0 = \frac{1}{2}(\mu_{H,D}(E) + \mu_{H,D}(F)) - \frac{\Delta_{H,D}(E) - \Delta_{H,D}(F)}{\mu_{H,D}(E) - \mu_{H,D}(F)} \qquad \rho_0^2 = (\mu_{H,D}(E) - s)^2 - 2\Delta_{H,D}(E).
\]

The next lemma is standard.

\begin{lemma}\label{lemma:destab_subsheaf}
Let $(X,H)$ be a smooth polarized surface and $E$ a coherent sheaf.
\begin{enumerate}
	\item If $A \to E$ is a destabilizing subobject of $E$ in $\mathcal{A}_s$, then $A$ is a sheaf.
	\item If $E[1] \to Q$ is a destabilizing quotient object of $E[1]$ in $\mathcal{A}_s$, then $Q=Q'[1]$ for some sheaf $Q'$.
\end{enumerate}
\begin{proof}
Both statements easily follow from taking the long exact sequence in cohomology.
\end{proof}
\end{lemma}

The method of \cite{feyzbakhsh} is to consider the exact sequence $0 \to E(-C) \to E \to i_\ast E|_C \to 0$ where $i:C \hookrightarrow X$ is an integral curve. There is a corresponding distinguished triangle in the derived category:
\[
E \to i_\ast E|_C \to E(-C)[1] \to E[1].
\]
If $E$ and $E(-C)[1]$ are $\sigma$-stable of the same phase, then $i_\ast E|_C$ is $\sigma$-semistable. The next lemma says that this is sufficient for $E|_C$ to be slope semistable.

\begin{lemma}\label{lemma:bridgelandstable_implies_slopestable}
Let $(X,H)$ be a smooth polarized surface, $i:C \hookrightarrow X$ an integral curve, and let $F$ be a torsion-free sheaf on $C$. Fix a stability condition $\sigma = (Z_{s,t},\mathcal{A}_s)$. If $i_\ast F$ is $\sigma$-(semi)stable, then $F$ is slope (semi)stable.
\end{lemma}
\begin{proof}
Applying Lemma \ref{lemma:ch_of_pushforward}, we obtain:
\[
\mu_{s,t}(i_\ast F) = \frac{\ch_2(i_\ast F)-(D+sH)\cdot \ch_1(i_\ast F)}{H \cdot \ch_1(i_\ast F)}.
\]
Let $r$ denote the rank of $F$ and let $F' \subset F$ be a subsheaf of rank $r'$. By Lemma \ref{lemma:ch_of_pushforward}, we have:
\[
\frac{\deg(F') - \frac{r'C^2}{2}-(D+sH)\cdot r'C}{H \cdot r'C} \leq \frac{\deg(F) - \frac{rC^2}{2}-(D+sH)\cdot rC}{H \cdot rC},
\]
if $F$ s $\sigma$-semistable, and strict inequality if $F$ is $\sigma$-stable. The result follows immediately because every torsion sheaf is an object in $\mathcal{A}_s$.
\end{proof}

The following important consequence appears in \cite{feyzbakhsh}.

\begin{cor}\label{lem:stableimplies_stableres}
Let $(X,H)$ be a smooth, polarized surface, $i:C \hookrightarrow X$ an integral curve, and $E$ an $(H,D)$-twisted stable sheaf on $X$. Let $\sigma = (Z_{s,t},\mathcal{A}_s) \in W(E,E(-C)[1])$ be a stability condition such that $E$ and $E(-C)[1]$ are both $\sigma$-stable. Then $E|_C$ is slope stable.
\end{cor}
\begin{proof}
Since $\sigma$ lies on the wall $W(E,E(-C)[1])$, we have that $\mu_{s,t}(i_\ast E|_C) = \mu_{s,t}(E) = \mu_{s,t}(E(-C)[1])$. By Lemma \ref{lemma:bridgelandstable_implies_slopestable}, we see immediately that $E|_C$ is slope semistable. By \cite[Prop. 16]{martinez}, the condition $\sigma$ may be perturbed to some $\sigma'$ so that $E$, $E(-C)[1]$, and $i_\ast E|_C$ are all $\sigma'$-stable. The result now follows from Lemma \ref{lemma:bridgelandstable_implies_slopestable}.
\end{proof}

\section{Stable restrictions on smooth surfaces}\label{sec:all_surfaces}
Let $X$ be a smooth surface and $H$ an ample divisor on $X$. Suppose $E$ is a $\mu_{H,D}$-stable sheaf on $X$ and $C \subset X$ an integral curve. The goal of this section is to give sufficient criteria for the restriction $E|_C$ to be stable on $C$. To prove Theorem \ref{thm:restriction_allsurfaces}, we would like to apply Lemma \ref{lem:stableimplies_stableres}. To do so, note that $E$ and $E|_C$ fit into an exact sequence, $0 \to E(-C) \to E \to i_\ast E|_C \to 0$, and there is correspondingly a distinguished triangle in the derived category $\mathcal{D}^b(X)$:
\[
E \to i_\ast E|_C \to E(-C)[1] \to E[1].
\]
If we can find a stability condition $\sigma=(Z_{s,t},\mathcal{A}_s)$ such that $E$ and $E(-C)[1]$ are $\sigma$-stable of the same slope, it will follow that $E|_C$ is slope stable. The set of $s,t$ such that $E$ and $E(-C)[1]$ have the same slope is given by a semicircular wall in the $(H,D)$-slice with center
\[
s = \frac{C \cdot \ch_1^D(E)}{\ch_0(E) H\cdot C} - \frac{C^2}{2 H \cdot C}.
\]

For any stable Chern character $\vec{v}$, there exists a wall $W_\vec{v}$, called the \emph{Giesker wall} bounding the chamber consisting of stability conditions $\sigma$ for which every $(H,D)$-twisted Gieseker stable sheaf of Chern character $\vec{v}$ is $\sigma$-stable. The Gieseker wall was computed in \cite{ch-nef} for $\Delta_{H,D}(E) \gg 0$. The method of proof of Theorem \ref{thm:restriction_allsurfaces} may be thought of as giving a rough approximation of the Gieseker wall.

We will show that $E$ cannot be destabilized by a subobject $A$ unless the radius of $W(A,E)$ is smaller than an explicit bound. Recall that by Lemma \ref{lemma:destab_subsheaf}, any such destabilizing object must be a sheaf. Case (1) of the next lemma first appears in \cite[Lemma 6.3]{abch} for $X=\P^2$. The proof in the general situation is essentially identical.

\begin{lemma}\label{lemma:bound_on_giesker_wall}
Let $E$ be an $(H,D)$-twisted Gieseker semistable sheaf on $X$ of positive rank, and suppose $A$ is a coherent sheaf.
\begin{enumerate}
\item If $A \to E$ is a map of coherent sheaves which is an inclusion of $\sigma$-semistable objects of the same slope for some $\sigma = (Z_{s,t}, \mathcal{A}_s)$ in the $(H,D)$-slice with $(s,t) \in W(A,E)$, then $A$ and $E$ are in $\mathcal{Q}_{s'}$ for all $(s',t') \in W(A,E)$.
\item If $E \to Q$ is a map of coherent sheaves such that $E[1] \to Q[1]$ is a surjection of $\sigma$-semistable objects of the same slope for some $\sigma = (Z_{s,t}, \mathcal{A}_s)$ in the $(H,D)$-slice with $(s,t) \in W(E[1],Q[1])$, then $E$ and $Q$ are in $\mathcal{F}_{s'}$ for all $(s',t') \in W(E[1],Q[1])$.
\end{enumerate}
\end{lemma}

We record a final technical lemma before proving the theorem.

\begin{lemma}\label{lemma:deltaE_larger}
Let $E$ be a vector bundle of rank $r$ on $X$ and $C$ any curve. Then there exists a line bundle $L$ such that $\mu_{H,D}(E \otimes L) > 0$ and
\[
r(r-1)\Delta_{H,D}(E\otimes L)-\mu_{H,D}(E\otimes L) \geq r(r-1)\Delta_{H,D}(E(-C) \otimes L)-\mu_{H,D}(E(-C)\otimes L).
\]

\end{lemma}
\begin{proof}
By the Hodge Index Theorem, we may write $c_1(E) = eH + \varepsilon$, $[C] = cH + \gamma$, $D = dH + \delta$, and $c_1 L = lH + \lambda$, where $\varepsilon, \gamma, \delta, \lambda \in H^\perp$ and $\varepsilon^2, \delta^2, \gamma^2, \lambda^2 \leq 0$. By taking $l \gg 0$, the inequality $\mu_{H,D}(E\otimes L) > 0$ is immediately satisfied.

If $\gamma = 0$, then $C$ is a multiple of $H$ and we may take $\lambda = 0$. Suppose $\gamma \neq 0$. Lemma \ref{lemma:change_of_deltas} gives:
\begin{align*}
	H^2\Delta_{H,D}(E \otimes L) &= \Delta(E) - \frac{1}{2r^2}\left( r\delta - \varepsilon - r \lambda\right)^2\\
	H^2\Delta_{H,D}(E \otimes L \otimes \O(-C)) &= \Delta(E) - \frac{1}{2r^2}\left(r\delta - \varepsilon - r \lambda + r \gamma\right)^2
\end{align*}
Thus,
\[
\Delta_{H,D}(E\otimes L) - \Delta_{H,D}(E \otimes L \otimes \O(-C)) = \frac{1}{2}\left(\gamma^2 + 2\gamma \cdot \delta - \frac{2\gamma \cdot \varepsilon}{r} - 2\gamma \cdot \lambda\right).
\]
Since $\gamma^2 < 0$, we may take $L = lH+ m\gamma$ for $m \gg 0$ and $l$ fixed.
\end{proof}

\begin{thm}\label{thm:restriction_allsurfaces}
Suppose $(X,H)$ is a smooth polarized surface, $C$ an integral curve on $X$, and $E$ a $\mu_{H,D}$-stable sheaf of rank $r \geq 2$ on $X$. Then $E|_C$ is stable if
\[
\frac{C^2}{2H \cdot C} > r(r-1)\Delta(E)+\frac{1}{2r(r-1)H^2}.
\]
\end{thm}
\begin{proof}
Since twisting $E$ by a line bundle does not affect its slope stability or the that of $E|_C$, by Lemma \ref{lemma:deltaE_larger} we may assume without loss of generality that $\mu_{H,D}(E) > 0$ and
\[
r(r-1)\Delta_{H,D}(E)-\mu_{H,D}(E) \geq r(r-1)\Delta_{H,D}(E(-C)[1])-\mu_{H,D}(E(-C)[1]).
\]

Take $D$ as in Lemma \ref{lemma:change_of_deltas} with $H \cdot D = 0$ and observe that the inequality in the statement of the theorem is then equivalent to
\begin{equation}\label{inequality:res_thm}
\frac{C^2}{2H \cdot C} - \frac{\ch_1^D(E) \cdot C}{r H \cdot C} + \mu_{H,D}(E) > r(r-1)H^2\Delta_{H,D}(E)+\frac{1}{2r(r-1)H^2}.
\end{equation}

We show that if (\ref{inequality:res_thm}) holds, then $E$ and $E(-C)[1]$ are $\sigma$-stable for stability conditions $\sigma$ lying on $W(E,E(-C)[1])$. Assume for a contradiction that there is a torsion-free sheaf $A$ and a map $A \to E$ destabilizing $E$ in $\mathcal{A}_s$. There are two possibilities: $\ch_0(A) < r$ and $\ch_0(A) \geq r$.

Assume first that $\ch_0(A) < r$. Let $(s_1,0)$ and $(s_2,0)$ be the endpoints of $W(A,E)$ with $s_1 < s_2$. If $(s_0,0)$ is the center of $W(A,E)$, then $(s_2-s_0)^2 = (\mu_{H,D}(E) - s_0)^2 - 2\Delta_{H,D}(E)$. Solving for $s_0$, we have:
\[
s_0 = \frac{1}{2}(\mu_{H,D}(E) + s_2) - \frac{\Delta_{H,D}(E)}{\mu_{H,D}(E) - s_2}.
\]
By Lemma \ref{lemma:bound_on_giesker_wall}, $s_2 \leq \mu_{H,D}(A) \leq \mu_{H,D}(E) - 1/(r(r-1)H^2)$, and so the wall $W(A,E)$ is bounded by the wall with center
\begin{equation}\label{eqn:center_1stcase}
s_0' = \mu_{H,D}(E) - \frac{1}{2r(r-1)H^2} - r(r-1)H^2\Delta_{H,D}(E)
\end{equation}

Suppose now that $\ch_0(A) \geq r$ and that $W(A,E)$ is the largest actual wall in the $(H,D)$-slice where $E$ is destabilized. Let
\[
0 \to A \to E \to Q \to 0
\]
be the destabilizing sequence in $\mathcal{A}_s$ with $s < \mu_{H,D}(E)$. We have that $A$ and $Q$ are semistable and consequently that $\Im(Z_{s,t}(A)), \Im(Z_{s,t}(Q)) > 0$. Since $\Im(Z_{s,t}) = H \cdot \ch_1^{D+sH}$, we conclude that
\begin{equation}\label{inequality:mainthm_c1_bound}
H\cdot\ch_1^D(A) > sH^2 \qquad \text{and} \qquad H\cdot\ch_1^D(Q) > sH^2.
\end{equation}
There are two possibilities to consider: either $(0,t)$ lies outside all actual walls for $\vec{v}$ for all $t>0$ or there is a point $(0,t)$ on $W(A,E)$. In the former case, the wall $W(A,E)$ is then bounded by the wall with center
\begin{equation}\label{eqn:center_2ndcase}
s_0'' = \mu_{H,D}(E) - \frac{\Delta_{H,D}(E)}{\mu_{H,D}(E)}.
\end{equation}
On the other hand, if $(0,t) \in W(A,E)$ for some $t$, then $A$ and $Q$ must destabilize $E$ in $\mathcal{A}_0$. By Inequality (\ref{inequality:mainthm_c1_bound}), we conclude that $
H \cdot \ch_1^D(A) > 0$ and $H \cdot \ch_1^D(Q) > 0$. Since we have assumed that $H\cdot\ch_1^D(E) > 0$, it follows that $0 < H \cdot \ch_1^D(A) < H \cdot \ch_1^D(E)$. Since we have chosen $D$ as in Lemma \ref{lemma:change_of_deltas}, we have $H \cdot D = 0$ and thus $\ch_1^D(E) = \ch_1(E)$ and $\ch_1^D(A) = \ch_1(A)$. Moreover, $\mu_{H,D}(A) < \mu_{H,D}(E)$, and therefore:
\[
\mu_{H,D}(A) \leq \mu_{H,D}(E) - \frac{1}{rH^2}.
\]
Repeating the arguments above, we conclude that $W(A,E)$ is bounded by the wall with center
\begin{equation}\label{eqn:center_3rdcase}
s_0''' = \mu_{H,D}(E) - \frac{1}{2rH^2} - rH^2\Delta_{H,D}(E)
\end{equation}

Comparing the centers in Equations (\ref{eqn:center_1stcase}), (\ref{eqn:center_2ndcase}), and (\ref{eqn:center_3rdcase}), we see that the center in Equation (\ref{eqn:center_1stcase}) corresponds to the largest wall. On the other hand, the center $s$ of $W(E, E(-C)[1])$ is
\[
s = \frac{C \cdot \ch_1^D(E)}{rH \cdot C} - \frac{C^2}{2 H \cdot C} < \mu_{H,D}(E)-r(r-1)H^2\Delta_{H,D}(E) - \frac{1}{2r(r-1)H^2}  = s_0'
\]
by our assumption. It follows that if Inequality (\ref{inequality:res_thm}) holds, then $E$ must be $\sigma$-stable for any stability condition $\sigma = (Z_{s,t},\mathcal{A}_s)$ with $(s,t)\in W(E,E(-C)[1])$.

We next show that $E(-C)[1]$ is $\sigma$-stable for conditions $\sigma$ on $W(E,E(-C)[1])$. We repeat the above argument for $E(-C)[1]$ using part (2) of Lemma \ref{lemma:bound_on_giesker_wall} to show that if
\[
\frac{C^2}{2H \cdot C} - \frac{\ch_1^D(E) \cdot C}{r H \cdot C} + \mu_{H,D}(E(-C)) > r(r-1)H^2\Delta_{H,D}(E(-C))+\frac{1}{2r(r-1)H^2},
\]
then $E(-C)[1]$ is $\sigma$-semistable along $W(E,E(-C)[1])$. This is automatically satisfied by our assumption. Thus $E|_C$ is slope stable by Lemmas \ref{lemma:bridgelandstable_implies_slopestable} and \ref{lem:stableimplies_stableres} if Inequality (\ref{inequality:res_thm}) is satisfied, which is equivalent to the inequality in the theorem by our choice of $D$.
\end{proof}

\subsection{General sheaves on $\P^2$}
Let $\vec{v}$ be a stable Chern character on $\P^2$ and $C \subset \P^2$ any integral curve. In this section we give sufficient criteria for the restriction $E|_C$ of a general element of $M(\vec{v})$ to be slope stable on $C$. The method of proof is similar to the above in that we exploit the distinguished triangle $E \to E|_C \to E(-C)[1] \to E[1]$. The main difference is that we require the wall $W(E,E(-C)[1])$ to be outside the so-called \emph{effective wall}, beyond which the general Gieseker stable vector bundle is $\sigma$-stable. The effective wall for $M_{\P^2}(\vec{v})$ was computed in \cite{chw}. To bound the effective wall, we need a few definitions. For details on stable sheaves on $\P^2$ we refer to \cite{lepotier} and \cite{chw}.

An \emph{exceptional bundle} $E$ on $\P^2$ is a (Gieseker) stable vector bundle such that $\Ext^1(E,E) =0$. A rational number $\nu$ is called an \emph{exceptional slope} if it is the slope of an exceptional bundle. If $\nu$ is an exceptional slope, then there is a unique exceptional bundle of slope $\nu$. Exceptional bundles are precisely the stable bundles $E$ with $\Delta(E) < \frac{1}{2}$.

Exceptional bundles are important for the classification of stable bundles on $\P^2$. If $\nu$ is an exceptional slope, let $\Delta_\nu$ denote the discriminant of the unique exceptional bundle of slope $\nu$ and put
\[
x_0(\nu) = \frac{3- \sqrt{5 + 8 \Delta_\nu}}{2}.
\]
Let $I_\nu$ be the open interval $I_\nu = (\nu-x_0(\nu), \nu+x_0(\nu))$ and let $P(x) = \frac{1}{2}(x^2+3x+2)$ be the Hilbert polynomial of $\O_{\P^2}$. We define the function $\delta(\mu)$ as
\[
\delta(\mu) = P(-|\mu-\nu|) - \Delta_\nu, \text{ if } \mu\in I_\nu.
\]

By \cite{drezet}, every slope $\mu$ of a sheaf lies in some interval $I_\nu$. A theorem of Dr\'ezet and Le Potier \cite{drezet-lepotier} says that for a stable Chern character $\vec{v}$ with $\Delta(\vec{v}) \geq \delta(\mu(\vec{v}))$, the moduli space $M(\vec{v})$ is a normal, irreducible, factorial projective variety of the expected dimension $r^2(2\Delta(\vec{v}) - 1)+1$. Furthermore, if $\Delta(\vec{v}) > \delta(\mu(\vec{v}))$, then $\pic(M(\vec{v}))$ is a free abelian group of rank 2 \cite{drezet}. We will assume for the rest of this section that $\vec{v}$ is such that $M(\vec{v})$ has Picard rank 2.

\begin{definition}[\cite{chw}]
If $\vec{v}$ is a Chern character, we define the \emph{associated parabola} $Q_\vec{v}$ in the $(\mu,\Delta)$-plane as the equation
\[
P(\mu+\mu(\vec{v})) - \Delta(\vec{v}) = \Delta.
\]
\end{definition}

\begin{thm}[\protect{\cite[Thm. 3.1]{chw}}]\label{thm:chw_munaught}
The parabola $Q_{\vec{v}}$ intersects the line $\Delta = \frac{1}{2}$ at two points. If $\mu_0(\vec{v}) \in \R$ is the larger of the two slopes such that $(\mu_0(\vec{v}), \frac{1}{2})\in Q_\vec{v}$, then there is a unique exceptional slope $\nu=\nu(\vec{v})$ such that $\mu_0(\vec{v}) \in (\nu-x_0(\nu), \nu+x_0(\nu)$.
\end{thm}

The slope $\mu_0(\vec{v})$ is computed explicitly in \cite{chw} and is given by the formula
\[
\mu_0(\vec{v}) = \frac{-3 - 2\mu(\vec{v}) + \sqrt{5+8\Delta(\vec{v})}}{2}.
\]

The unique exceptional bundle from the previous theorem is called the \emph{corresponding exceptional bundle}, and its Chern character is called the \emph{corresponding exceptional character}.

\begin{definition}
Given two coherent sheaves $E$ and $F$ on any smooth surface $X$, we define the pairing
\[
\chi(E,F) = \sum_{i=1}^2 (-1)^i \ext^i(E,F).
\]
By the Riemann-Roch formula, $\chi(E,F)$ depends only the Chern characters of $E$ and $F$ and therefore also defines a pairing on $\knum(X)\otimes \R$. If $\vec{v}$ is a Chern character, then $\vec{v}^\ast$ will denote the \emph{dual Chern character}: if $\vec{v} = (r, c_1, \ch_2)$, then $\vec{v}^\ast = (r, -c_1, \ch_2)$.
\end{definition}

\begin{definition}\label{def:exceptional_characters}
Let $F$ be a stable sheaf with Chern character $\vec{v}$, and let $\vec{w}$ be the corresponding exceptional character. Define the \emph{corresponding orthogonal invariants} $\mu^+(\vec{v})$ and $\Delta^+(\vec{v})$ as follows.
\begin{enumerate}
	\item $(\mu^+(\vec{v}), \Delta^+(\vec{v})) = Q_\vec{v} \cap Q_{\vec{w}^\ast}$ if $\chi(\vec{v}^\ast,\vec{w}) > 0$.
	\item $(\mu^+(\vec{v}), \Delta^+(\vec{v})) = (\mu(\vec{w}),\Delta(\vec{w}))$ if $\chi(\vec{v}^\ast,\vec{w}) = 0$.
	\item $(\mu^+(\vec{v}), \Delta^+(\vec{v})) = Q_\vec{v} \cap Q_{\vec{w}^\ast-3}$ if $\chi(\vec{v}^\ast,\vec{w}) < 0$, where $\vec{w}^\ast-3 = (\rk(\vec{w}), - c_1(\vec{w})-3, \ch_2(\vec{w}))$.
\end{enumerate}
\end{definition}

Let $E \in M(\vec{v})$ be general, and suppose $E^+$ is a sheaf with slope $\mu^+(\vec{v})$, discriminant $\Delta^+(\vec{v})$, and rank $r^+$, where $r^+$ is sufficiently large and divisible. Then for stability conditions $\sigma \in W(E,E^+)$, there is an exact sequence
\[
0 \to E^+ \otimes \Hom(E^+, E) \to E \to W \to 0
\]
in the category $\mathcal{A}_s$, where $W$ is the mapping cone of the evaluation map $E^+ \otimes \Hom(E^+, E) \to E$. By computing the Gieseker walls for $E^+$ and $W$, we can show that $E$ is $\sigma$-stable for $\sigma \in W(E,E^+)$, and thereby compute the effective wall for $E$. This is carried out in \cite{chw}.

By \cite[Thm. 5.7]{chw}, the center of the effective wall is $s_0 = -\mu^+(E) - 3/2$. To prove Theorem \ref{thm:restriction_p2}, it therefore suffices to show that the center of $W(E,E(-C)[1])$ is at most $s_0$. The center of $W(E,E(-C)[1])$ is $\mu(E) - d/2$, where $d$ is the degree of the curve $C$. To produce sufficient criteria to ensure $\mu(E) - d/2 < s_0$, we estimate $\mu^+(E)$ in the next lemma.

\begin{lemma}
Let $\vec{v}$ be a stable Chern character on $\P^2$. Then
\[
\mu^+(\vec{v}) \leq \frac{\sqrt{5+8\Delta(\vec{v})}}{2}+1 -\mu(\vec{v})
\]
for $\Delta(\vec{v}) \gg 0$.
\begin{proof}
Let $\vec{w}$ be the corresponding exceptional character. We begin by producing different bounds for $\mu^+(\vec{v})$ depending on $\chi(\vec{v}^\ast, \vec{w})$. To simplify notation, let $r = \rk(\vec{v})$, $\mu = \mu(\vec{v})$, $\mu_0=\mu_0(\vec{v})$ as in Theorem \ref{thm:chw_munaught}, $\Delta = \Delta(\vec{v})$, $\mu^+ = \mu^+(\vec{v})$, and $\Delta^+ = \Delta^+(\vec{v})$. Also set $r' = \rk(\vec{w})$, $\mu' = \mu(\vec{w})$, $\Delta' = \Delta(\vec{w})$, and $x_0 = x_0(\vec{w})$ as the invariants of the exceptional character $\vec{w}$.

We claim the following inequalities hold:
\begin{enumerate}
	\item If $\chi(\vec{v}^\ast,\vec{w}) > 0$, then
	\[
		\mu^+ \leq \frac{2\Delta}{\sqrt{1 + 8 \Delta}-3}-\frac{3}{2}+\frac{\sqrt{5+8\Delta} - \sqrt{5}}{4}-\mu.
	\]
	\item If $\chi(\vec{v}^\ast,\vec{w}) \leq 0$, then
	\[
		\mu^+ \leq \frac{\sqrt{5+8\Delta}-\sqrt{5}}{2} - \mu.
	\]
\end{enumerate}
If $\chi(\vec{v}^\ast,\vec{w}) > 0$, then $(\mu^+, \Delta^+) = Q_\vec{v} \cap Q_{\vec{w}^\ast}$. Direct computation shows that
\[
	\mu^+ = \frac{\Delta-\Delta'}{\mu+\mu'}-\frac{3}{2}+\frac{\mu'-\mu}{2}.
\]
Since $\chi(\vec{v}^\ast , \vec{w}) > 0$, the Riemann-Roch formula implies $rr'\left(P(\mu+\mu')-\Delta-\Delta'\right) > 0$, or
\[
(\mu+\mu')^2+3(\mu+\mu')+2 > 2(\Delta+\Delta').
\]
Using that $\mu+\mu' >0$ \cite[Lemma 3.6]{chw}, we have
\[
\mu+\mu' > \frac{\sqrt{8(\Delta + \Delta')+1}-3}{2}.
\]
Moreover, since $\mu'$ is in the interval $[\mu_0-x_0, \mu_0+x_0]$, we see that $\mu'\leq \mu_0+x_0$, or
\[
\mu'\leq \frac{\sqrt{5+8\Delta}-\sqrt{5+8\Delta'}}{2}-\mu
\]
By the Bogomolov inequality and the fact that $\vec{w}$ is exceptional, we have $\Delta \geq \frac{1}{2} > \Delta' \geq 0$. We obtain the following estimates:
\[
\frac{\Delta - \Delta'}{\mu + \mu'} \leq \frac{2\Delta}{\sqrt{1 + 8 \Delta}-3},
\]
and
\[
\mu' - \mu \leq \frac{\sqrt{5+8\Delta} - \sqrt{5}}{2}-2\mu.
\]
Combining these with the formula for $\mu^+$, we obtain (1).

Suppose now that $\chi(\vec{v}^\ast,\vec{w}) \leq 0$. Then $\mu^+$ lies in the interval $[\mu_0 - x_0, \mu_0 + x_0]$ (see \cite{chw}). Thus $\mu^+ \leq \mu_0+x_0$, which immediately gives (2). Direct computation now shows that for $\Delta \gg 0$, the right hand sides of Inequalities (1) and (2) are dominated by
\[
\mu^+(\vec{v}) \leq \frac{\sqrt{5+8\Delta(\vec{v})}}{2}+1 -\mu(\vec{v}).
\]
\end{proof}
\end{lemma}

Note that the assumption ``$\Delta(\vec{v}) \gg 0$'' in the above lemma is fairly mild. In fact, it suffices to have $\Delta(\vec{v}) \geq 4$. To see this, observe that the inequality
\[
\frac{2\Delta}{\sqrt{1 + 8 \Delta}-3}-\frac{3}{2}+\frac{\sqrt{5+8\Delta} - \sqrt{5}}{4}<\frac{\sqrt{5+8\Delta}}{2}+1
\]
holds as soon as $16 \Delta^6 - 32 \Delta^5 - 44 \Delta^4 - 56 \Delta^3 - 59 \Delta^2 - 20 \Delta + 4 > 0$. The largest root of this polynomial is $\Delta \approx 3.3$. Moreover, the only issue that occurs is that the sum $\mu(\vec{v}) + \mu(\vec{w})$ can tend to zero in the case $\chi(\vec{v}^\ast, \vec{w}) > 0$, which in turn causes $\mu^+(\vec{v})$ to grow very large. By \cite[Lemma 3.6]{chw}, $\mu(\vec{v}) + \mu(\vec{w})$ is strictly positive when $\chi(\vec{v}^\ast,\vec{w}) > 0$, and so in practice this computation can be made sharp.

\begin{thm}\label{thm:restriction_p2}
Suppose $C \subset \P^2$ is a degree $d$ integral curve and $E$ is a general slope stable vector bundle on $\P^2$ such that the moduli space $M_{\P^2}(E)$ has Picard rank 2. Then $E|_C$ is stable if $d > \sqrt{5+8\Delta(E)}+5$ and $\Delta \geq 4$.
\end{thm}
\begin{proof}
Let $\vec{v}$ denote the Chern character of $E$. If $d > \sqrt{5+8\Delta(E)}+5$ and $\Delta \geq 4$, then by the above lemma, $\mu(\vec{v}) - d/2 < -\mu^+(\vec{v}) - 3/2$. This shows that for the general $E \in M(\vec{v})$, the wall $W(E,E(-C)[1])$ lies outside the effective wall for $E$. Arguing as in Theorem \ref{thm:restriction_allsurfaces}, $E$ and $E(-C)[1]$ are $\sigma$-stable for conditions lying on $W(E,E(-C)[1])$. By Lemma \ref{lem:stableimplies_stableres}, it follows that $E|_C$ is stable.
\end{proof}

\subsection{Extending stable vector bundles} Rather than restricting from surfaces to curves, we can ask the opposite question: when does a stable vector bundle on $C$ extend to a stable vector bundle on $X$? We will assume $X$ is a smooth surface and $\vec{v}$ is a stable Chern character on $X$ such that $\Delta(\vec{v})$ is large enough that $M_{X,(H,D)}(\vec{v})$ is well-behaved:

\begin{thm*}[O'Grady \cite{ogrady}]
Let $(X,H)$ be a smooth polarized surface, and $\vec{v}$ a Chern character with $r(\vec{v}) > 0$. If $\Delta(\vec{v}) \gg 0$, then $M_{X,(H,D)}(\vec{v})$ is normal, generically smooth, irreducible, and nonempty of the expected dimension $2r^2\Delta(\vec{v})-(r^2-1)\chi(\O_X)$. Furthermore, slope stable sheaves are dense in $M_{X,(H,D)}(\vec{v})$.
\end{thm*}

Let $C \subset X$ be a smooth curve in an ample class $dA$ (with $A$ not necessarily a multiple of the polarization $H$) such that the restriction of a general element of $M_{X,H}(\vec{v})$ to $C$ is stable. We have a restriction map $M_{X,H}(\vec{v}) \dashrightarrow U_C(r,c_1\cdot C)$, where $U_C(r,e)$ denotes the moduli space parameterizing semistable rank $r$, degree $e$ torsion-free sheaves on $C$.

The next proposition will show that for $d \gg 0$, the restriction map is generically finite. Note that $\dim U_C(r,e) = r^2(g-1)+1$, where $g$ is the genus of $C$, and $\dim M(\vec{v}) = r^2(2\Delta(\vec{v}) - 1) + 1$. By adjunction, $g$ grows with $d^2$, thus the image of $M_{X,H}(\vec{v})$ in $U_C(r,c_1\cdot C)$ is typically of large codimension.

\begin{prop}\label{prop:extending_vb}
With the above notation, if $d \gg 0$, then the dimension of the image of the restriction map equals the dimension of $M_{X,H}(\vec{v})$.
\end{prop}
\begin{proof}
Let $M_0 \subset M_{X,H}(\vec{v})$ denote the open, smooth subset of vector bundles $E$ with $\Ext^2(E,E) = 0$. Let $E \in M_0$ be general and consider the differential $T_{E}M_0 \to T_{E|_C}U_C(r,c_1\cdot C)$. Since $T_{E}M_0 = \Ext^1(E,E)$ and $T_{E|_C}U_C(r,c_1\cdot C) = \Ext^1_C(E|_C,E|_C)$, we have a map $\delta:\Ext^1(E,E) \to \Ext^1_C(E|_C,E|_C)$.

Next we apply the functor $\Hom(E,E\otimes -)$ to the short exact sequence $0 \to \O(-C) \to \O \to \O_C \to 0$. In the associated long exact sequence in cohomology, there is a map $\Ext^1(E,E) \to \Ext^1(E,E|_C)$. Now,
\[
\Ext^1(E,E|_C) = H^1(X,E^\vee \otimes E|_C) = H^1(C,E^\vee|_C\otimes E|_C) = \Ext^1_C(E|_C,E|_C).
\]
Thus the map $\Ext^1(E,E) \to \Ext^1(E,E|_C)$ in the long exact sequence is precisely the map $\delta$ (see, e.g., \cite[Prop. 2.6]{ch-hirzebruch}). Because $E$ is stable, $\hom(E,E(-C)) = 0$. Moreover,
\[
\ext^1(E,E(-C)) = h^1(E^\vee \otimes E(-C)) = h^1(E^\vee \otimes E\otimes \omega_X \otimes \O(dA)),
\]
by Serre duality. Thus for $d\gg 0$, this group vanishes and the map $\delta$ is injective.
\end{proof}

\section{Cohomology}\label{sec:cohomology}
If $E$ is a stable vector bundle on the smooth surface $X$ and $C\subset X$ a curve, then we can use the standard exact sequence
\[
0 \to E (-C) \to E \to i_\ast E|_C \to 0
\]
to compute the cohomology of $E|_C$, at least when the cohomology of $E$ is well-understood. The work of G\"ottsche-Hirschowitz on $\P^2$ \cite{goettsche-hirschowitz} and Coskun-Huizenga on Hirzebruch surfaces \cite{ch-hirzebruch} allows us to understand the cohomology of general elements of $M_{X,(H,D)}(\vec{v})$. After restricting to the curve $C$, we obtain some results that are interesting in the context of higher-rank Brill-Noether theory. In particular, we show that the number $h^0(C,E|_C)$ can violate the expected dimension count of the Brill-Noether number. In light of Proposition \ref{prop:extending_vb}, we conclude that there is a fairly large-codimensional subspace of $U_C(r,e)$ consisting of vector bundles with unexpectedly many global sections.

The \emph{Brill-Noether number}, denoted $\rho_{r,e}^k$ is the expected dimension of the subvariety $B_{r,e}^k$ of $U_C(r,e)$ consisting of rank $r$, degree $e$ stable sheaves with at least $k$ global sections. It is given by the following formula:
\[
\rho_{r,e}^k = r^2(g-1)+1 - k(k-e+r(g-1)) = \dim U_C(r,e) - k(k-\chi(E)),
\]
where $g$ is the genus of $C$.

\subsection{Plane curves}
We will exploit the following result of G\"ottsche-Hirschowitz that describes the cohomology for a general stable sheaf. For this subsection, $X = \P^2$.

\begin{thm*}[G\"ottsche-Hirschowitz \cite{goettsche-hirschowitz}] The general sheaf $E \in M(\vec{v})$ of rank $r \geq 2$ has at most one nonzero cohomology group:
\begin{enumerate}
	\item If $\chi(E) \geq 0$ and $\mu(E) > -3$, then $H^1(E) = H^2(E) = 0$.
	\item If $\chi(E) \geq 0$ and $\mu(E) \leq -3$, then $H^0(E)=H^1(E) = 0$.
	\item If $\chi(E) < 0$, then $H^0(E) = H^2(E) = 0$.
\end{enumerate}
\end{thm*}

We will make frequent use of the following Riemann-Roch calculation:

\begin{lemma}
Suppose $E$ is a rank $r$, degree $e$ stable bundle on $\P^2$, and let $C$ be a smooth curve of degree $d$. Then,
\begin{enumerate}
	\item $\chi(E) = r + 3e/2 + \ch_2(E)$.
	\item $\chi(E(-C)) = r + (3/2)(e-dr)+\ch_2(E) - de + rd^2/2$.
\end{enumerate}
\end{lemma}

Resolving $E$ via the sequence $0 \to E(-C) \to E \to i_\ast E|_C \to 0$ we are led to consider nine cases: $H^i(E) \neq 0$, $H^j(E(-C)) \neq 0$ for $i,j \in \{0,1,2\}$. In all but two cases, $H^0(E|_C) =0$ or $H^1(E|_C) =0$, and so there are no unexpected global sections. The two interesting cases are when $H^0(E) \neq 0$, $H^2(E(-C)) \neq 0$ and $H^1(E) \neq 0$, $H^1(E(-C)) \neq 0$. In the latter case, a direct computation shows that these bundles are Brill-Noether general when $\deg(C)$ is large:

\begin{prop}
Suppose $E$ is a general element of $M(\vec{v})$ such that $H^1(E) \neq 0$ and $H^1(E(-C)) \neq 0$ for a smooth degree $d$ curve $C$. Let $e = \deg(E)$ and $k=h^0(E|_C)$. Then $\rho_{r,de}^k \geq 0$ for $d \gg 0$.
\end{prop}
\begin{proof}
From the long exact sequence in cohomology we see that $h^0(E|_C) \leq h^1(E(-C))$ and $h^1(E|_C) \leq h^1(E)$. Thus,
\begin{align*}
\rho^k_{r,de} &\geq r^2(g-1)+1 + \left(r+\frac{3e}{2}+\ch_2(E)\right)\left( -r-\frac{3}{2}(e-dr)-\ch_2(E)+de-\frac{rd^2}{2}\right)\\
&= r^2\left(\frac{(d-1)(d-2)}{2}-1\right)+1+ \left(r+\frac{3e}{2}+\ch_2(E)\right)\left( -r-\frac{3}{2}(e-dr)-\ch_2(E)+de-\frac{rd^2}{2}\right)\\
&=\frac{d^2}{2}\left(\frac{-er}{2}-\ch_2(E)\right)+P(d),
\end{align*}
where $P(d)$ is a polynomial of degree one in $d$. Since $-er/2-\ch_2(E) = -\chi(E) + r$ and $r-\chi(E) > 0$, we see that for $d \gg 0$ we must have $\rho_{r,de}^k \geq 0.$
\end{proof}

For the other case, we have:

\begin{prop}
Let $C$ be a smooth plane curve of degree $d$, and suppose $E$ is a general element of $M(\vec{v})$ such that $H^0(E) \neq 0$ and $H^2(E(-C)) \neq 0$. Suppose further that $\chi(E) > r$. Let $e = \deg(E)$ and $k = h^0(E|_C)$. Then $\rho^{k}_{r,de}<0$ but $B^k_{r,de}$ is nonempty for $d \gg 0$.
\end{prop}
\begin{proof}
By the long exact sequence in cohomology, it follows that $\rho^k_{r,de} < 0$ if
\[
r^2(g-1) +1 < \left(r+\frac{3e}{2}+\ch_2(E)\right)\left(r+\frac{3e}{2}+\ch_2(E) - de + r(g-1)\right).
\]
Since $g=(d-1)(d-2)/2$, the left-hand side is dominated by the term $d^2r^2/2$ and the right-hand side is dominated by the term $d^2(r^2/2 + 3er/4+r\ch_2(E)/2)$. Thus the inequality holds for large $d$ if $r^2/2 + 3er/4 + r\ch_2(E)/2 > r^2/2$. This is equivalent to the inequality $\chi(E) > r$. Note that $\mu(E(-C)) = \mu(E) - d$, so the hypotheses are preserved by taking $d$ large. The restriction $E|_C$ is semistable for $d\gg 0$, thus the locus $B^{k}_{r,de}$ is nonempty.
\end{proof}

Note that if $r$ and $\ch_2(E)$ are fixed, then $h^0(E)$ grows with $e$ and thus the hypotheses of the above theorem hold for fixed $\ch_2(E)$ with $d \gg 0$ and $e \gg 0$.

\subsection{Curves on Hirzebruch surfaces}
Coskun-Huizenga \cite{ch-hirzebruch} computed the cohomology of a general stable sheaf on a Hirzebruch surface. Let $X = \mathbb{F}_m$ be the Hirzebruch surface $X = \P(\O_{\P^1} \oplus \O_{\P^1}(m))$, $m >0$. Let $M$ denote the curve class of self-intersection $-m$ and $F$ the class of a fiber. If $\vec{v}=(r,c_1,\ch_2)$ is a stable Chern character on $X$, then define $v(\vec{v}) = c_1/r$. The cohomology of the general element $E \in M_{X,(H,D)}(\vec{v})$ can be determined with the intersection numbers $v(\vec{v}) \cdot M$ and $v(\vec{v}) \cdot F$ in addition to the Euler characteristic $\chi(\vec{v})$.

\begin{thm}[\protect{\cite[Theorem 3.1]{ch-hirzebruch}}] Let $\vec{v}$ be a stable Chern character on $\mathbb{F}_m$ with positive rank and $E \in M_{X,(H,D)}(\vec{v})$ a general stable sheaf. Then
\begin{enumerate}
	\item If $v(\vec{v}) \cdot F \geq -1$, then $h^2(E) = 0$.
	\item If $v(\vec{v}) \cdot F \leq -1$, then $h^0(E) = 0$.
	\item If $v(\vec{v}) \cdot F = -1$, then $h^1(E) = -\chi(\vec{v})$ and all other cohomology vanishes.
\end{enumerate}
Suppose now $v(\vec{v}) \cdot F > -1$. Then either of the numbers $h^0(E)$ or $h^1(E)$ determines the Betti numbers of $E$:
\begin{enumerate}
\setcounter{enumi}{3}
	\item If $v(\vec{v}) \cdot M \geq -1$, then $E$ has at most one nonzero cohomology group. If $\chi(\vec{v}) \geq 0$ then $h^0(E) = \chi(\vec{v})$, and if $\chi(\vec{v}) \leq 0$, then $h^1(E) = - \chi(\vec{v})$.
	\item If $v(\vec{v}) \cdot M < -1$, then $H^0(E) \cong H^0(E(-M))$.
\end{enumerate}
If $v(\vec{v}) \cdot F < -1$ and $r\geq 2$, then the cohomology of $E$ can be determined by Serre duality.
\end{thm}

\begin{lemma}
Let $X = \mathbb{F}_m$. Then:
\begin{enumerate}
	\item $K_X = -2M-(e+2)F$, and if $C$ is a smooth curve of class $aM + bF$, then $g(C) = \frac{1}{2}(1-a)(am-2b+2).$
	\item Let $H = aM + bF$. Then $H$ is ample if and only if $a>0$ and $b>am$. If $H$ is ample, $C$ a curve of class $dH$, and $E$ a $\mu_H$-stable sheaf of rank $r \geq 2$, then $E|_C$ is semistable if
\[
d > 2r(r-1)\Delta(E)+\frac{1}{2r(r-1)(ab-a^2m)}.
\]
\end{enumerate}
\begin{proof}
Statement (1) is a direct adjunction computation, and statement (2) is Theorem \ref{thm:restriction_allsurfaces}.
\end{proof}
\end{lemma}

In the same spirit as the $\P^2$ case, we are interested in understanding the asymptotic behavior of $h^0(E|_C)$ when $a \to \infty$ and when $b \to \infty$. Note that we have the following:
\[
\begin{rcases}
v(E(-C)) \cdot M &\to \infty\\
\chi(E(-C)) &\to -\infty
\end{rcases}
\text{ as }a \to \infty,
\qquad
\begin{rcases}
v(E(-C)) \cdot M &\to -\infty\\
\chi(E(-C)) &\to \infty
\end{rcases}
\text{ as }b \to \infty
\]

As with $\P^2$, the interesting cases occur when $E|_C$ has both $H^0$ and $H^1$. Because of the behavior of the Betti numbers of $E(-C)$ as $a$ and $b$ grow independently, there are two possibilities: Either $E$ has only $H^0$ and $E(-C)$ has both $H^1$ and $H^2$, or $E$ and $E(-C)$ have only $H^1$. The latter case can only occur when $[C] = adM + bdF$ with $v(E(-C)) \cdot M \geq -1$, $\chi(E(-C)) < 0$, and $a+b \gg 0$. As with the analogous case on $\P^2$, this does not produce any stable bundles on $C$ with unexpectedly many global sections. On the other hand, we have:

\begin{prop}
Let $H = aM + bF$ be ample, and let $E$ be a general element of $M_{X,(H,D)}(\vec{v})$ for some stable Chern character $\vec{v}$ such that $v(\vec{v}) \cdot F < 0$ and $E$ has only $H^0$. Let $C$ be a smooth curve of class $dH$, and put $k=h^0(C,E|_C)$, $e = \deg(E|_C)$. If $\chi(E) > r$ and $a \geq 2$, then $\rho^{k}_{r,e} < 0$, but $B^k_{r,e}$ is nonempty for $b,d \gg 0$.
\end{prop}
\begin{proof}
Write $c_1(E) = xM + yF$. Then $x<0$ by assumption and $v(E(-C))\cdot F = x/r-ad < -1$. Thus $h^0(E(-C)) = 0$. Since $h^2(E|_C) = 0$, we have $h^0(E|_C) \geq h^0(E) = \chi(E)$. It therefore suffices to show
\begin{equation}\label{inequality:hirzebruch}
r^2\left(g-1)\right)+1 < \chi(E)\left(\chi(E)-e+rg\right).
\end{equation}
Since $g = \frac{1}{2}(1-ad)(adm-2db+2)$, we see that $g \to \infty$ as $b\to \infty$. Then $e = d(-amx + bx + ay)$ with $x < 0$. Since we have also assumed that $\chi(E) > r$, the right side dominates the left for $b \gg 0$. Further, as $b$ grows, $v(E(-C)) \cdot F$ is unchanged and $h^0(E(-C)) = 0$. We may therefore find $b$ sufficiently large so that Inequality (\ref{inequality:hirzebruch}) holds for all $d$. Since $E|_C$ is semistable for large $d$, the claim follows.
\end{proof}

\bibliographystyle{plain}

\end{document}